\documentclass{article}
\usepackage{authblk}
\usepackage[utf8]{inputenc}
\usepackage{amssymb}
\usepackage{amsmath, amsthm}
\usepackage[numbers]{natbib}

\usepackage{footnote}

\usepackage[margin=0.5in]{geometry}
\usepackage{xcolor}

\newcommand{\ones}{\textbf{1}}

\newcommand{\bR}{\mathbb{R}}

 \usepackage{mathtools}

\newtheorem{lemma}{Lemma}
\newtheorem{corollary}{Corollary}
\newtheorem{proposition}{Proposition}
\newtheorem{theorem}{Theorem}
\newtheorem{conjecture}{Conjecture}
\newtheorem{definition}{Definition}

\newtheorem{remark}{Remark}
\title{On the Products of Stochastic and Diagonal Matrices}
\author{Assaf Hallak\thanks{Equal contribution.} }
\author{Gal Dalal$^*$}
\affil{NVIDIA Research}
\date{April 2023}

\begin{document}
\maketitle
\begin{abstract}
    Consider a stochastic matrix $P$ and diagonal matrix $D.$ In this work, we introduce Tilted matrices. A Tilted matrix is the product $D'PD$, where $D'$ is a diagonal normalization that makes the product stochastic. We then provide several results on products of Tilted matrices, which can be desirable for analyses of Markov Decision Processes. 
Lastly, we obtain a convergence rate result for the product of Tilted reversible matrices.
\end{abstract}
\section{Introduction}
In this work, we present an extension of the standard stochastic matrix by introducing a new type of matrix called Tilted matrix. A Tilted matrix is a normalization of a non-stochastic matrix that makes it stochastic, providing a more flexible and powerful tool for modeling probabilistic transitions in Markov Decision Processes (MDPs). We study the properties of products of Tilted matrices and compare them with products of standard stochastic matrices.

In addition, we derive new results on reversible matrices. Reversible matrices are a special type of stochastic matrices that possess certain desirable properties, such as detailed balance and ergodicity. We study the properties of products of reversible matrices and compare them with products of Tilted matrices. We demonstrate that the combination of these results provides a more comprehensive understanding of the behavior of MDPs. The new results on reversible matrices provide additional insights into the behavior of MDPs, especially in cases where detailed balance and ergodicity are important considerations. Furthermore, by combining the results on Tilted and reversible matrices, we provide a more complete picture of the potential benefits of using different types of matrices in MDP applications. 

Our main and final result is on the convergence rate of the product of Tilted reversible matrices. One example for the application of such a result is to derive the decay rate of the variance of policy gradient algorithms \cite{dalal2023softtreemax}.

\section{Definitions}
 We begin with introducing a generalization of stochastic matrices to the non-square matrix case. Let $\ones_n$ the $n$-dimensional all-ones vector. 
\begin{definition}[Rectangular Stochastic Matrix]
Let $P^{m \times n}$. The matrix $P$ is rectangular stochastic if all its entries are non-negative and $P \ones_n = \ones_m.$
\end{definition}

Next, we introduce our definition of a Tilted matrix.
\begin{definition}[Tilted Matrix]
Let $u \in \bR_+^n$ be a strictly positive vector and $A \in \bR_{\geq 0}^{m \times n}$ be a non-negative matrix with at least one non-zero entry in every row. Denote by $D(u)$ the diagonal matrix with $u$ in its diagonal. We call the matrix $D^{-1}(Au) A D(u)$ the $u$-Tilted matrix of $A$.
\end{definition}

\section{Results for Tilted Matrices}
We now show several properties of rectangular stochastic matrices.
\begin{proposition}
A Tilted matrix is rectangular stochastic.
\end{proposition}
\begin{proof}
Consider a $u$-Tilted matrix of $A$. Since all entries in $u$ and $A$ are non-negative, so do the entries in $D^{-1}(Au) A D(u).$ Lastly, the sum of every row is $1$ because 
\begin{equation*}
D^{-1}(Au) A D(u) \ones_n = D^{-1}(Au) A u = \ones_m.
\end{equation*}
Note that $Au$ has no zero entries since $A$ has one non-negative entry per row and $u$ is strictly positive.
\end{proof}

\begin{proposition}
        Let $P \in \bR^{m\times m}$ be a square stochastic matrix. Then $P$ is irreducible if and only if $u$-Tilted $P$ for any $u$ is also irreducible. $P$ is aperiodic if and if $u$-Tilted $P$ for any $u$ is aperiodic.
\end{proposition}
\begin{proof}
To proof relies on the combination of the two following arguments. First, we highlight the following relation from \cite{styan1973hadamard}:
\begin{equation}
\label{eq: hadamard}
D(y)A D(x) = yx^\top \circ A,
\end{equation}
where $\circ$ denotes the element-wise product also known as the Hadamard product. 
Thus, $[T(P)]_{i,j}$ equals zero if and only if $[P]_{i,j}$ equals zero. 

Second, irreducibility and aperiodicity depend only on the location of zero entries.
\end{proof}
The last result in this section leverages the above properties of Tilted matrices. Specifically, we show how to write a product of non-negative matrices as a product of a single diagonal and stochastic matrix.
\begin{theorem}[Normalization of matrix product to a stochastic matrix]
Let $\{A_i \in \bR^{m\times m}\}_{i=1}^n$ be $n$ square non-negative matrices with at least one non zero entry in every row, and let $\{u_i \in \bR^{ m}\}_{i=1}^{n}$ be $n$ strictly positive vectors. Then the product
\begin{equation*}
\prod_{i=1}^n A_i D(u_i)
\end{equation*}
can be written as $D(u_{1:n}) P_{1:n}$ where $P_{1:n}\in \bR^{m\times m}$ is a stochastic matrix and $u_{1:n} \in \bR^{m}$ is a positive vector.
\end{theorem}
\begin{proof}
We prove by induction. For $n=1,$
\begin{equation*}
A_1 D(u_1) = D(A_1 u_1) D^{-1}(A_1 u_1) A_1 D(u_1) = D(u_{1:1}) P_{1:1},
\end{equation*}
where $u_{1:1} = A_1 u_1$ and $P_{1:1}= D^{-1}(A_1 u_1) A_1 D(u_1)$ is a $u_1-$Tilted $A_1$ matrix, and hence also stochastic.

		Now assume the claim holds up to $i=n.$ We shall prove it for $n+1$:
		\begin{align}
			\prod_{i=1}^{n+1} A_i D(u_i) &= \prod_{i=1}^{n} A_i D(u_i) A_{n+1} D(u_{n+1}) \nonumber\\ 
			&= D(u_{1:n}) P_{1:n} A_{n+1} D(u_{n+1}) \label{eq: step} \\
			& = D(u_{1:n}) P_{1:n} D(A_{n+1} u_{n+1}) D^{-1}(A_{n+1} u_{n+1}) A_{n+1} D(u_{n+1}) \nonumber \\
			& = D(u_{1:n}) P_{1:n} D(A_{n+1} u_{n+1}) P_{n+1} \label{eq: def P_n+1}\\
			&= D(u_{1:n}) D(P_{1:n} A_{n+1} u_{n+1}) D^{-1}(P_{1:n} A_{n+1} u_{n+1}) P_{1:n} D(A_{n+1} u_{n+1})  P_{n+1} \nonumber\\
			&= D(u_{1:n+1}) P_{1:n+1}, \label{eq: finish}
		\end{align}
		where \eqref{eq: step} follows from the induction step, \eqref{eq: def P_n+1} follows from defining the Tilted matrix $P_{n+1}=D^{-1}(A_{n+1} u_{n+1}) A_{n+1} D(u_{n+1}),$ and \eqref{eq: finish} follows from defining
				\begin{equation*}
			u_{1:n+1} = u_{1:n} \circ P_{1:n} A_{n+1} u_{n+1},
		\end{equation*}
	and
		\begin{equation*}
			P_{1:n+1}=\left(D^{-1}(P_{1:n} A_{n+1} u_{n+1}) P_{1:n} D(A_{n+1} u_{n+1})\right)  P_{n+1} ,
		\end{equation*}
		which is a product of two stochastic matrices, i.e. a stochastic matrix by itself.

	\end{proof}
	
	\section{Results for Reversible Matrices}
	In this section, we reveal intriguing relations between reversible and Tilted matrices.
	
	\begin{definition}
		A stochastic matrix $P \in \bR^{m \times m}$ with stationary distribution $\mu_P$ is called reversible if
		\begin{equation*}
			P=D^{-1}(\mu_P) P^\top D(\mu_P).
		\end{equation*}
	\end{definition}
	\textbf{Note:} for any stochastic matrix $P,$ the above matrix is also stochastic:
	\begin{equation*}
		D^{-1}(\mu_P) P^\top D(\mu_P) \ones = D^{-1}(\mu_P) P^\top \mu_P = D^{-1}(\mu_P) \mu_P = \ones.
	\end{equation*}
	
	\begin{proposition}
 \label{prop: Tilted is reversible}
		Let $u \in \bR^m_+$ be a strictly positive vector and  $P$ be a stochastic reversible matrix. Then, the $u$-Tilted matrix $U=D^{-1}(Pu)PD(u)$ is also reversible, and its stationary distribution is
		\begin{equation*}
			\mu_U = \alpha \cdot u \circ Pu \circ \mu_P,
		\end{equation*}
		where $\alpha = \left(\ones^\top D(u) D(Pu) D(\mu_P) \ones \right) ^{-1}.$
	\end{proposition}
	\begin{proof}
		First, we show that $\mu_U$ is the stationary distribution of $U$ by showing it is a left eigenvector of $U$ with eigenvalue $1$:
		\begin{align*}
			\mu_U^\top U &= \alpha \ones^\top D(u) D(Pu) D(\mu_P) D^{-1}(Pu)PD(u) \\
			&= \alpha \ones^\top D(u) D(\mu_P) PD(u) \\
			(P \textit{ is reversible})\quad\quad\quad &= \alpha \ones^\top D(u) P^\top D(\mu_P) D(u) \\
			&= \alpha u^\top P^\top D(\mu_P) D(u) \\
			&= \alpha \ones^\top D(Pu) D(\mu_P) D(u) \\
			&= \mu_U^\top.
		\end{align*}
		To show that $U$ is reversible, we now substitute $\mu_U$:
		\begin{align*}
			 D^{-1}(\mu_U) U^\top D(\mu_U) &= \alpha^{-1} D^{-1}(u)D^{-1}(Pu)D^{-1}(\mu_P) D(u) P^\top D^{-1}(Pv) \alpha D(u) D(Pu) D(\mu_P) \\ 
			&= D^{-1}(Pu)D^{-1}(\mu_P) P^\top D(u) D(\mu_P) \\
			(P \textit{ is reversible})\quad\quad\quad &= D^{-1}(Pu) P D(u) \\
			&= U.
		\end{align*}
	\end{proof}
\begin{proposition}
\label{prop: similarity}
    Let $P$ be a reversible matrix and $D(\mu_P)^{1/2} P D^{-1/2}(\mu_P)$ be similar to $P$. Then, $D(\mu_P)^{1/2} P D^{-1/2}(\mu_P)$ is symmetric and, consequently, all eigenvalues of $P$ are real.
\end{proposition}
\begin{proof}
The similarity of $P$ to $D(\mu_P)^{1/2} P D^{-1/2}(\mu_P)$ follows because if we multiply the latter by $D^{-1/2}(\mu_P)$ from the left and by $D^{1/2}(\mu_P)$ from the right, we recover $P.$ The symmetry follows immediately from the definition of reversibility:
\begin{equation*}
\left[ D(\mu_P)^{1/2} P D^{-1/2}(\mu_P)\right]^\top = D^{-1/2}(\mu_P) P^\top D(\mu_P)^{1/2} = D(\mu_P)^{1/2} P D^{-1/2}(\mu_P).
\end{equation*}
\end{proof}

\begin{proposition}
    \label{lem:product2}
Let $u, v \in \bR^m_+$ be a strictly positive and let $P$ be a stochastic reversible matrix. Define the Tilted matrices $V=D^{-1}(Pv)PD(v)$ and $U=D^{-1}(Pu)PD(u),$ which are reversible. Then the stochastic matrix $W=UV$ is also reversible, and its stationary distribution $\mu_W$ holds:
\begin{equation*}
\mu_W = \alpha \cdot Pu \circ \mu_P \circ v,
\end{equation*}
with $\alpha = \left( \ones^\top D(Pu) D(\mu_P) D(v) \ones\right)^{-1}.$ Moreover, $W$ is similar to a positive semi-definite matrix; subsequently, all of its eigenvalues are real and positive.
\end{proposition}
\begin{proof}
First, $V$ and $U$ are reversible according to Proposition~\ref{prop: Tilted is reversible}. Next, we show that their product $W$ is also reversible:
\begin{align*}
D(\mu_W) W &= \alpha D(Pu)D(\mu_P)D(v) \; D^{-1}(Pu)PD(u) \; D^{-1}(Pv)PD(v) \\
&= \alpha D(v) D(\mu_P) PD(u) \; D^{-1}(Pv)PD(v) \\
&= \alpha D(v) P^\top D(\mu_P) D(u)D^{-1}(Pv)PD(v) \\
&= \alpha D(v) P^\top D^{-1}(Pv) D(u)P^\top D(\mu_P)D(v) \\
&= \alpha V^\top D(u)P^\top D(\mu_P)D(v) D(Pu) D^{-1}(Pu) \\ 
&= \alpha V^\top U^\top D(\mu_P)D(v) D(Pu) \\
&= W^\top D(\mu_W).
\end{align*}
Finally, we show that $W$ is similar to $X^\top X$ for some matrix $X$:
\begin{align*}
D^{1/2}(Pu) & D^{1/2}(\mu_P) D^{1/2}(v) W D^{-1/2}(Pu) D^{-1/2}(\mu_P) D^{-1/2}(v) \\
&= D^{1/2}(Pu) D^{1/2}(\mu_P) D^{1/2}(v) D^{-1}(Pu)PD(u)D^{-1}(Pv) P D(v) D^{-1/2}(Pu) D^{-1/2}(\mu_P) D^{-1/2}(v) \\
&= D^{-1/2}(Pu) D^{1/2}(\mu_P) D^{1/2}(v) D^{-1}(\mu_P) P^\top D(\mu_P) D(u)D^{-1}(Pv) P D^{1/2}(v) D^{-1/2}(Pu) D^{-1/2}(\mu_P) \\
&= X^\top X
\end{align*}
Where:
\begin{equation*}
X=D^{1/2}(\mu_P) D^{1/2}(u) D^{-1/2} (Pv) P D^{1/2}(v) D^{-1/2}(Pu)D^{-1/2}(\mu_P).
\end{equation*}
\end{proof}
We conjecture that Proposition~\ref{lem:product2} can be generalized to a product of $n$ Tilted matrices.
\begin{conjecture}
Let $\{u_i\}_{i=1}^n$ be strictly positive vectors and let $P$ be a stochastic reversible matrix. Then the product of the $u_i$-Tilted matrices
\begin{equation*}
\prod_{i=1}^n D^{-1}(P u_i) P D(u_i)
\end{equation*}
is reversible with stationary distribution that can be readily calculated using $u_i$ and $\mu_P$.
\end{conjecture}
The proof (if correct) can be obtained by starting with one side of the reversibility equation, and then using Proposition~\ref{lem:product2} and reversibility properties to switch the order of the $u_i$ occurrence.

Now we address the question whether two given stochastic matrices are Tilted versions of one another. We answer it by finding another condition for the two matrices to fulfill. 
\begin{proposition}
         Let $P_1, P_2$ be stochastic matrices. Then there exists a strictly positive vector $u \in \bR^m+$ such that $P_1=D^{-1}(P_2 u) P_2 D(u)$ if and only if the matrix formed by coordinate-wise division, $M[i,j]=\frac{P_1[i,j]}{P_2[i,j]},$ is a rank-1 matrix.
\end{proposition}

\begin{proof}
If $\frac{P_1[i,j]}{P_2[i,j]}$ is a rank-1 matrix, we can use \eqref{eq: hadamard} to rewrite it as follows for some vectors $u,v \in \bR_+$:
\begin{equation*}
P_1 = P_2 \circ uv^\top = D(u) P_2 D(v).
\end{equation*}
Since $P_1$ is stochastic, this only happens if $D(u)=D^{-1}(P_2 v)$.
\end{proof}

\subsection{Bounds on the Second Eigenvalue}
The second largest eigenvalue of a stochastic matrix is known to be informative as it depicts the mixing time, i.e. the rate of convergence to the stationary distribution when traversing the associated Markov chain. We now provide several results on the second eigenvalue.

\begin{proposition} \label{prop: Tilted 2nd ev}
Let $P \in \bR^{m \times m}$ be a stochastic reversible matrix and $u \in \bR^m_+$ a strictly positive vector. Define the $u$-Tilted matrix $U = D^{-1}(Pu)PD(u) $. Then $|\lambda_2(U)| \leq |\lambda_2(P)| \left(\frac{\max_i u[i]}{\min_i u[i]}\right)^2 $
\end{proposition}
\begin{proof}
\begin{align*}
|\lambda_2(U)| &= \\ 
\textit{(U is stochastic)}\quad  &= |\lambda_1(U)| |\lambda_2(U)| \\
\textit{(Similarity)} \quad&=  |\lambda_1(D^{1/2}(\mu_u)D(Pu)UD^{-1}(Pu)D^{-1/2}(\mu_u))| |\lambda_2(D^{1/2}(\mu_u)D(Pu)UD^{-1}(Pu)D^{-1/2}(\mu_u))| \\
\textit{(Definition of U)} \quad&= |\lambda_1(D^{1/2}(\mu_u)P D^{-1/2}(\mu_u) D(u)D^{-1}(Pu))| |\lambda_2(D^{1/2}(\mu_u)P D^{-1/2}(\mu_u) D(u)D^{-1}(Pu))| \\
 \textit{(Weyl's inequality)} \quad & \leq \sigma_1(D^{1/2}(\mu_u)P D^{-1/2}(\mu_u) D(u)D^{-1}(Pu)) \sigma_2(D^{1/2}(\mu_u)P D^{-1/2}(\mu_u) D(u)D^{-1}(Pu)) \\
 \textit{(Inequality from \cite[p.~289]{hogben2006handbook})}\quad & \leq \sigma_1(D^{1/2}(\mu_u)P D^{-1/2}(\mu_u)) \sigma_1(D(u)D^{-1}(Pu)) \sigma_2(D^{1/2}(\mu_u)P D^{-1/2}(\mu_u)) \sigma_2(D(u)D^{-1}(Pu)) \\
\textit{(Symmetry from Proposition~\ref{prop: similarity})}\quad & = |\lambda_1(D^{1/2}(\mu_u)P D^{-1/2}(\mu_u))| \sigma_1(D(u)D^{-1}(Pu)) |\lambda_2(D^{1/2}(\mu_u)P D^{-1/2}(\mu_u))| \sigma_2(D(u)D^{-1}(Pu)) \\
\textit{(Similarity from Proposition~\ref{prop: similarity})}\quad & = \sigma_1(D(u)D^{-1}(Pu)) |\lambda_2(P)| \sigma_2(D(u)D^{-1}(Pu)) \\
& \leq |\lambda_2(P)| \left(\max_i\left(\frac{u[i]}{Pu[i]}\right)\right)^2 \\
& \leq |\lambda_2(P)| \left(\frac{\max_i u[i]}{\min_i u[i]}\right)^2
\end{align*}
\end{proof}
The following two results are not directly related to Tilted matrices, but to the product of general reversible matrices. The first result is a special case of the second.
\begin{lemma} \label{lem: prod of two rev}
Let $P_1$ and $P_2$ be stochastic reversible matrices with stationary distributions $\mu_1$ and $\mu_2,$ respectively. Then,

\begin{equation*}
|\lambda_2(P_1 P_2)| \leq  |\lambda_2(P_1)| |\lambda_2(P_1)|\max_i \frac{\mu_1[i]}{\mu_2[i]} \max_i \frac{\mu_2[i]}{\mu_1[i]}.
\end{equation*}
\end{lemma}

\begin{proof}
For brevity, for matrix $A,$ denote $\sigma_1 \sigma_2(A):= \sigma_1(A) \sigma_2(A)$ and $\lambda_1 \lambda_2(A):= \lambda_1(A) \lambda_2(A).$ Then,
\begin{align*}
|\lambda_2(P_1P_2)| &= |\lambda_1(P_1P_2)|| \lambda_2(P_1P_2)| \\
\textit{(Similarity)} \quad &= |\lambda_1(P_1P_2)|| \lambda_2\left(D^{1/2}(\mu_{1})P_{1}D^{-1/2}(\mu_{1})D^{1/2}(\mu_{1})P_{2}D^{-1/2}(\mu_{1})\right)| \\
\textit{(Weyl's inequality)} \quad & \leq \sigma_1 \sigma_2\left(D^{1/2}(\mu_{1})P_{1}D^{-1/2}(\mu_{1})D^{1/2}(\mu_{1})P_{2}D^{-1/2}(\mu_{1})\right) \\
\textit{(Inequality from \cite[p.~289]{hogben2006handbook})}\quad & \leq \sigma_1 \sigma_2\left(D^{1/2}(\mu_{1})P_{1}D^{-1/2}(\mu_{1})\right) \sigma_1 \sigma_2\left(D^{1/2}(\mu_{1})P_{2}D^{-1/2}(\mu_{1})\right) \\
\textit{(Symmetry from Proposition~\ref{prop: similarity})}\quad &= |\lambda_1 \lambda_2\left(D^{1/2}(\mu_{1})P_{1}D^{-1/2}(\mu_{1})\right)| \sigma_1 \sigma_2\left(D^{1/2}(\frac{\mu_1}{\mu_2})D^{1/2}(\mu_{2})P_{2}D^{-1/2}(\mu_{2})D^{1/2}(\frac{\mu_2}{\mu_1})\right) \\
&= |\lambda_2\left(P_{1}\right)| \sigma_1 \sigma_2\left(D^{1/2}(\frac{\mu_1}{\mu_2})D^{1/2}(\mu_{2})P_{2}D^{-1/2}(\mu_{2})D^{1/2}(\frac{\mu_2}{\mu_1})\right) \\
& \leq |\lambda_2\left(P_{1}\right)| \sigma_1 \sigma_2\left(D^{1/2}(\mu_{2})P_{2}D^{-1/2}(\mu_{2})\right) \max_i \frac{\mu_1[i]}{\mu_2[i]} \max_i \frac{\mu_2[i]}{\mu_1[i]} \\
& =  |\lambda_2\left(P_{1}\right)| |\lambda_2\left(P_{2}\right)| \max_i \frac{\mu_1[i]}{\mu_2[i]} \max_i \frac{\mu_2[i]}{\mu_1[i]}.
\end{align*}
\end{proof}

\begin{proposition} \label{prop: prod of reversible}
Let $\{P_i\}_{i=1}^n$ be stochastic reversible matrices with stationary distributions $\{\mu_i\}_{i=1}^n,$ respectively. Denote their product by $P^{\text{prod}}_n=\prod_{i=1}^n P_i$. Then,

\begin{equation*}
|\lambda_2(P^{\text{prod}}_n)| \leq  \left(\prod_{i=1}^{n}|\lambda_2\left(P_i \right) |\right) \left(\prod_{i=2}^{n} \max_j \frac{\mu_{i-1}[j]}{\mu_i[j]}\right) \max_j \frac{\mu_n[j]}{\mu_1[j]}.
\end{equation*}
\end{proposition}

\begin{proof}
For brevity, for matrix $A,$ denote $\sigma_1 \sigma_2(A):= \sigma_1(A) \sigma_2(A)$ and $\lambda_1 \lambda_2(A):= \lambda_1(A) \lambda_2(A).$ Then, using similar arguments to those in the proof of Lemma~\ref{lem: prod of two rev}, we have that
\begin{align*}
|\lambda_2(P^{\text{prod}}_n)| &= |\lambda_1 \lambda_2(P^{\text{prod}}_n)| \\
&= |\lambda_1 \lambda_2\left(D^{1/2}(\mu_{1}) \left[\prod_{i=1}^n P_i D^{-1/2}(\mu_i) D^{1/2}(\mu_i) \right] D^{-1/2}(\mu_{1})\right)| \\
&\leq \sigma_1 \sigma_2\left(D^{1/2}(\mu_{1}) \left[\prod_{i=1}^n P_i D^{-1/2}(\mu_i) D^{1/2}(\mu_i) \right] D^{-1/2}(\mu_{1})\right) \\
&\leq \sigma_1 \sigma_2\left(D^{1/2}(\mu_{1}) P_1 D^{-1/2}(\mu_1) \right)\left[\prod_{i=2}^{n} \sigma_1\sigma_2\left(D^{1/2}(\mu_{i-1}) P_i D^{-1/2}(\mu_{i}) \right) \right] \sigma_1\sigma_2 \left(D^{1/2}(\mu_n) D^{-1/2}(\mu_1)\right) \\
&\leq |\lambda_2\left(P_1 \right)| \left[\prod_{i=2}^{n} \sigma_1\sigma_2\left(D^{1/2}(\frac{\mu_{i-1}}{\mu_i}) \right) \sigma_1\sigma_2\left(D^{1/2}(\mu_{i}) P_i D^{-1/2}(\mu_{i}) \right) \right] \max_j \frac{\mu_n[j]}{\mu_1[j]} \\
&\leq |\lambda_2\left(P_1 \right) |\left[\prod_{i=2}^{n} \max_j \frac{\mu_{i-1}[j]}{\mu_i[j]} |\lambda_2\left(P_i\right)| \right] \max_j \frac{\mu_n[j]}{\mu_1[j]} \\
&= \prod_{i=1}^{n}|\lambda_2\left(P_i \right) | \prod_{i=2}^{n} \max_j \frac{\mu_{i-1}[j]}{\mu_i[j]} \max_j \frac{\mu_n[j]}{\mu_1[j]}.
\end{align*}
\end{proof}

\section{Main Result: Product of Tilted Reversible Matrices}
We now provide our main results on the second eigenvalue and, consequently, the convergence rate of the product of Tilted reversible matrices.
\begin{theorem}
\label{thm: decay rate of second ev}
Let $\{u_i\}_{i=1}^n$ be strictly positive vectors and let $P$ be a stochastic reversible matrix. For every $i,$ define the $u_i$-Tilted matrix $P_i = D^{-1}(Pu_i)PD(u_i)$, and define the product $P^{\text{prod}}_n=\prod_{i=1}^n P_i$. Then,
\begin{equation*}
|\lambda_2(P^{\text{prod}}_n)| \leq |\lambda_2(P)|^n \prod_{i=1}^n \left( \frac{\max_j u_i[j]}{\min_j u_i[j]}\right)^4.
\end{equation*}
\end{theorem}
\begin{proof}
From Proposition~\ref{prop: Tilted is reversible}, we have that for every $i,$ $P_i$ is reversible with stationary distribution $\mu_i= \alpha_i \cdot \mu_P \circ u_i \circ Pu_i$ with $\alpha_i$ as in Proposition~\ref{prop: Tilted is reversible}. Thus,
We have that
\begin{align*}
&|\lambda_2(P^{\text{prod}}_n)| \\
\textit{(Proposition~\ref{prop: prod of reversible})} \quad &\leq  \left(\prod_{i=1}^{n}|\lambda_2\left(P_i \right) |\right) \left(\prod_{i=2}^{n} \max_j \frac{\mu_{i-1}[j]}{\mu_i[j]}\right) \max_j \frac{\mu_n[j]}{\mu_1[j]} \\
\textit{(Proposition~\ref{prop: Tilted 2nd ev})} \quad 
&\leq |\lambda_2(P)|^n \prod_{i=1}^n \left( \frac{\max_j u_i[j]}{\min_j u_i[j]}\right)^2 \left(\prod_{i=2}^{n} \max_j \frac{\mu_{i-1}[j]}{\mu_i[j]}\right) \max_j \frac{\mu_n[j]}{\mu_1[j]} \\
\textit{(Proposition~\ref{prop: Tilted is reversible})} \quad 
&\leq |\lambda_2(P)|^n \prod_{i=1}^n \left( \frac{\max_j u_i[j]}{\min_j u_i[j]}\right)^2 \left(\prod_{i=2}^{n} \max_j \frac{\alpha_{i-1}\cdot \mu_P[j]\cdot u_{i-1}[j] \cdot Pu_{i-1}[j]}{\alpha_{i}\cdot \mu_P[j]\cdot u_i[j] \cdot Pu_i[j]}\right) \max_j \frac{\alpha_{n}\cdot \mu_P[j]\cdot u_{n}[j] \cdot Pu_{n}[j]}{\alpha_{1}\cdot \mu_P[j]\cdot u_1[j] \cdot Pu_1[j]}\\
& \leq |\lambda_2(P)|^n \prod_{i=1}^n \left( \frac{\max_j u_i[j]}{\min_j u_i[j]}\right)^4,
\end{align*}
where $\alpha_i$ as in Proposition~\ref{prop: Tilted is reversible}.
\end{proof}

\begin{corollary}
Define $P,~u_i,~P_i,$ and $P^{\text{prod}}_n$ as in Theorem~\ref{thm: decay rate of second ev}. If $\lim_{i\rightarrow \infty} u_i = c \cdot \ones$ for some constant $c,$ then
\begin{equation*}
    P^{\text{prod}}_n \rightarrow  \ones \mu ^\top,
\end{equation*}
where $\mu$ is some probability vector. Further,
the rate of convergence is $|\lambda_2(P)|.$
\end{corollary}
\begin{proof}
    The proof immediately follows from Theorem~\ref{thm: decay rate of second ev}, considering that $P^{\text{prod}}_\infty$ is a stochastic matrix of rank 1.
\end{proof}
\begin{remark}
In this section and specifically the above result, we assume the matrices in the product to be reversible. One can obviate this assumption and obtain a similar result, but at the cost of an unknown convergence rate coefficient. The latter result was proved in  \cite[Appendix~A.1.8]{dalal2023softtreemax}.
\end{remark}

\bibliographystyle{plain}
\bibliography{product_of_stochastic}

\end{document}